\documentclass[graybox]{svmult}

\usepackage{type1cm}        
\usepackage{makeidx}         
\usepackage{graphicx}        
\usepackage{multicol}        
\usepackage[bottom]{footmisc}

\newcommand{\Gr}{{\rm gph\,}}

\newcommand{\co}{{\rm conv\,}}

\newcommand{\J}{{\cal J}}

\newcommand{\Lsp}{{\cal L}}
\usepackage{newtxtext}       %
\usepackage{newtxmath}       
\newcommand{\Lb}{\bar\Lambda}
\newcommand{\Lbv}{\bar\Lambda(v)}
\newcommand{\lb}{{\bar\lambda}}
\newcommand{\R}{\mathbb{R}}

\newcommand{\mv}{\,\vert\, }
\newcommand{\xb}{\bar x}
\newcommand{\yb}{\bar y}
\newcommand{\zb}{\bar z}

\newcommand{\vb}{\bar v}

\newcommand{\KbG}{{\bar K_\Gamma}}
\newcommand{\Span}{\mathop{\rm span\,}\limits}
\newcommand{\argmax}{\mathop{\rm arg\,max\,}\limits}

\makeindex             

\begin{document}

\title*{Constraint qualifications and optimality conditions in bilevel optimization}
\author{Jane J. Ye}
\institute{Jane J. Ye \at Department of Mathematics and Statistics, University of Victoria, Canada,  \email{janeye@uvic.ca}. 
}
%
%
\maketitle


\abstract{In this paper we study constraint qualifications and optimality conditions for bilevel programming problems. We strive to derive checkable constraint qualifications in terms of problem data and applicable optimality conditions. For the bilevel program with convex lower level program we discuss  drawbacks of  reformulating  a bilevel programming problem by  the mathematical program with complementarity constraints and present a new sharp necessary optimality condition for the reformulation by the mathematical program with a generalized equation constraint. For the bilevel program with a nonconvex lower level program we propose a relaxed constant positive linear dependence (RCPLD) condition for the combined program.}

\section{Introduction}
\label{sec:1}
In this paper we onsider the following bilevel program:
\begin{eqnarray*}
({\rm BP})~~~~~~\min && F(x,y)\nonumber\\
{\rm s.t.} && y\in S(x), \ 
 G(x,y)\leq 0,\ H(x,y)= 0,
\end{eqnarray*}
where  $S(x)$ denotes the solution set of the lower level program
\begin{eqnarray*}
({\rm P}_x)~~~~~~~~
\min_{y\in \Gamma (x)} \  f(x,y),
\end{eqnarray*}
where $\Gamma(x):=\{y\in \mathbb{R}^m: g(x,y)\leq 0, h(x,y)=0\}$ is the feasible region of the lower level program, 
 $F:\mathbb{R}^n\times \mathbb{R}^m \rightarrow \mathbb{R}$, $G:\mathbb{R}^n\times \mathbb{R}^m \rightarrow \mathbb{R}^p$ and  $H:\mathbb{R}^n\times \mathbb{R}^m \rightarrow \mathbb{R}^q$ 
 $f:\mathbb{R}^n\times \mathbb{R}^n \rightarrow \mathbb{R}$, $g:\mathbb{R}^n\times \mathbb{R}^m \rightarrow \mathbb{R}^r, h:\mathbb{R}^n\times \mathbb{R}^m \rightarrow 
 \mathbb{R}^s$. Throughout the paper, for simplicity we assume that $S(x)\not =\emptyset$ for all $x$.

 In economics literature, a bilevel program is sometimes referred to as a Stackelberg game due to the introduction of the concept  by Stackelberg \cite{v}.  Although it can be used to model a game between the leader and the follower of a two level hierarchical system, the bilevel program has been used to model much wider range of  applications; see e.g. \cite{d1,d2}. Recently, it has been applied to  hyper-parameters selection  in machine learning; see e.g. \cite{K,L}. 
 
 The classical approach or the first order approach to study optimality conditions for bilevel programs is to replace the lower level problem by its Karush-Kuhn-Tucker (KKT) conditions and minimize over the original variables as well as the multipliers. The resulting problem is a so-called mathematical program with complementarity constraints or mathematical program with equilibrium constraints. The class of mathematical program with complementarity/equilibrium constraints has been studied intensively in the last three decades; see e.g. \cite{Luo-Pang-Ralph,Out-Koc-Zowe} and the reference within. 
 
 There are two issues involved in using the first order approach. Firstly, since the KKT condition is only a sufficient but not necessary condition for optimality, the first order approach can only be used when the lower level problem is a convex program. Secondly, even when the lower level is a convex program if the lower level problem has more than one multiplier, the resulting problem is not equivalent to the original bilevel program if local optimality is considered. In this paper we discuss these issues and present some strategies to deal with this problem. These strategies including using the value function approach, the combined approach and the generalized equation approach. 

For a stationary condition to hold at a local optimal solution, usually certain constraint qualifications are required to hold. There are some weak constraint qualifications which are not checkable since they are defined implicitly; e.g. Abadie constraint qualification. In this paper we concentrate on only those checkable constraint qualifications. 

 The following {notation} will be used throughout the paper.  We denote by  $B(\xb; \delta)$  the closed ball centered at $\xb$ with radius $\delta$ and by $B$ the closed unit ball centered at $0$.  We denote by  $\mathbb{B}_\delta(\xb)$  the open ball centered at $\xb$ with radius $\delta$. For a matrix
$A$, we denote by $A^T$ its transpose. 
The inner product of two vectors $x, y$ is denoted by
$x^T y$ or $\langle x,y\rangle$ and by
 $x\perp y$ we mean $\langle x, y\rangle =0$. 
The polar cone of a set $\Omega$ is
$\Omega^\circ=\{x|x^Tv\leq 0 \ \forall v\in \Omega\}$.
For a set $\Omega$, we denote by $\co \Omega$ the convex hull
 of $\Omega$. For a differentiable mapping $P:\mathbb R^d\rightarrow \mathbb R^s$, we denote by $\nabla P(z)$ the Jacobian matrix of $P$ at $z$ if $s>1$ and the gradient vector if $s=1$. For a function $f:\R^d \rightarrow \R$, we denote by $\nabla^2 f(\bar z)$ the Hessian matrix of $f$ at $\bar z$.  Let $M:\R^d\rightrightarrows\R^s$ be an arbitrary set-valued mapping. We denote its graph by $ {\rm gph}M:=\{(z,w)| w\in M(z)\}.$ $o:\R_+\rightarrow \R$ denotes a function with the property that $o(\lambda)/\lambda\rightarrow 0$ when $\lambda\downarrow 0$. By $z_{k}\stackrel{\Omega}{\to}z$ we mean that  $z_k\in \Omega $ and $z_k\rightarrow z$.
\section{Preliminaries on variational analysis}
In this section, we gather some preliminaries  in variational analysis and optimization theories that will be needed in the paper. The reader may find more details in the monographs \cite{Clarke,Mor,RoWe98} and  in the papers we refer to.

\begin{definition}[Tangent cone and normal cone]
Given a set
$\Omega\subseteq \mathbb R^d$ and a point $\bar z\in\Omega$,
the (Bouligand-Severi) {\em tangent/contingent cone} to $\Omega$
at $\bar z$ is a closed cone defined by
\begin{equation*}\label{normalcone}
T_\Omega(\bar z)
:=\limsup_{t\downarrow 0}\frac{\Omega-\bar z}{t}
=\Big\{u\in\mathbb R^d\Big|\;\exists\,t_k\downarrow
0,\;u_k\to u\;\mbox{ with }\;\bar z+t_k u_k\in\Omega ~~\forall ~ k \Big\}.
\end{equation*}
The (Fr\'{e}chet) {\em regular normal cone} and the (Mordukhovich) {\em limiting/basic normal cone} to $\Omega$ at $\bar
z\in\Omega$ are  closed cones defined by
\begin{eqnarray}
&& \widehat N_\Omega(\bar z):=(T_\Omega(\bar z))^\circ\nonumber\\
\mbox{and }  &&
N_\Omega(\bar z):=\left \{z^\ast \in \mathbb{R}^d\mv \exists z_{k}\stackrel{\Omega}{\to}\zb \mbox{ and } z^\ast_k\rightarrow z^\ast \mbox{ such that } z^\ast_{k}\in \widehat{N}_{\Omega}(z_k) \  \forall k \right \},
\nonumber
\end{eqnarray}
respectively.
\end{definition}
When  the set $\Omega$ is convex,  the regular and the limiting normal cones are equal and  reduce to
the classical  normal cone of convex analysis, 
 i.e.,
 $$N_\Omega(\bar z):=\{z^\ast| \langle z^*, z-\bar z\rangle \leq 0 \quad \forall z\in \Omega\} .$$

We now give definitions for subdifferentials.
\begin{definition}[Subdifferentials] Let $f:\mathbb{R}^d \rightarrow   \bar{ \mathbb{R}}$ be an extended value function, $\bar x\in \mathbb{R}^d$ and   $f(\bar x)$  is finite. The regular subdifferential of $f$ at $\xb$ is the set defined by
$$\widehat{\partial } f(\xb):=\{ v\in \mathbb{R}^d| f(x)\geq f(\xb) +\langle v, x-\xb\rangle +o (\|x-\xb\|) \}.$$
The limiting subdifferential of $f$ at $\xb$ is the set defined by
$$\partial f(\xb):=\{ v\in \mathbb{R}^d|  v=\lim_k  v_k, v_k\in  \widehat{\partial } f(x_k), x_k\rightarrow \xb, f(x_k)\rightarrow f(\xb) \}.$$
Suppose that $f$ is Lipschitz continuous at $\xb$. Then the Clarke subdifferential of $f$ at $\bar x$ is the set defined by
$$\partial^c f(\xb)=conv \partial f(\xb).$$
\end{definition}
When  the function $f$  is convex, all the subdifferentials defined above are equal and reduce to
the classical subgradient of convex analysis, i.e.,
$${\partial } f(\xb):=\{ v\in \mathbb{R}^d| f(x)\geq f(\xb) +\langle v, x-\xb\rangle \}.$$

\begin{definition}[Coderivatives] 	For a set-valued map $\Phi:\mathbb{R}^d \rightrightarrows \mathbb{R}^s$ and a point $(\bar x, \bar y)\in {\rm gph} \Phi$,
  the Fr\'{e}chet coderivative  of $\Phi$ at $(\bar x,\bar y)$ is a multifunction $\widehat D^*\Phi(\bar x,\bar y):\mathbb{R}^s \rightrightarrows\mathbb{R}^d$ defined as
    $$\widehat D^*\Phi(\bar x,\bar y)(w):=\left\{\xi\in\mathbb{R}^d|(\xi,-w)\in \widehat N_{{\rm gph} \Phi}(\bar x,\bar y) \right\}.$$
	And the limiting (Mordukhovich) coderivative of $\Phi$ at $(\bar x,\bar y)$ is a multifunction $D^*\Phi(\bar x,\bar y):\mathbb{R}^s
	\rightrightarrows \mathbb{R}^d$ defined as
	$$D^*\Phi(\bar x,\bar y)(w):=\left\{{\xi}\in\mathbb{R}^d|({\xi},-w)\in N_{{\rm gph} \Phi}(\bar x,\bar y) \right\}.$$
\end{definition}
We now review some concepts of stability of a set-valued map.

\begin{definition}[Aubin \cite{Aubin1984Lipschitz}] Let  $\Sigma : \R^{n}\rightrightarrows \R^{d}$ be a set-valued map and $(\bar{\alpha}, \bar x)\in {\hbox{gph}} \Sigma$. We say that $\Sigma$ is  pseudo-Lipschitz continuous    at  $(\bar{\alpha}, \bar x)$ if there exist a neighborhood $\mathbb{V}$ of $\bar{\alpha}$, a neighborhood $\mathbb{U}$ of $\bar x$ and $\kappa\geq 0$ such that
	{\rm{\begin{equation*}
			\Sigma \left(\alpha\right) \cap \mathbb{U} \subseteq \Sigma \left( \alpha^\prime \right) + \kappa \left\| \alpha^\prime - \alpha \right\|  B, \ \ \forall \alpha^\prime,\alpha\in \mathbb{V}.
			\end{equation*}}}
\end{definition}
\begin{definition}[Robinson
	\cite{Robinson1975Stability}] Let $\Sigma : \R^{n}\rightrightarrows \R^{d}$ be a set-valued map and  {$\bar{\alpha}\in \R^n$}. We say that $\Sigma$ is  upper-Lipschitz continuous    at  {$\bar{\alpha}$}  if there exist a neighborhood $\mathbb{V}$ of $\bar{\alpha}$ and $\kappa\geq 0$ such that
	{\rm{\begin{equation*}
			\Sigma \left(\alpha\right) \subseteq \Sigma \left(\bar {\alpha}\right) + \kappa \left\| \alpha - \bar{\alpha} \right\| B,\ \ \forall \alpha\in \mathbb{V}.
			\end{equation*}}}
\end{definition}
\begin{definition}[Ye and Ye \cite{YeYe}] Let $\Sigma : \R^{n}\rightrightarrows \R^{d}$ be a set-valued map and $(\bar{\alpha}, \bar x)\in {\hbox{gph}} \Sigma$. 
 We say that $\Sigma$ is   calm (or pseudo upper-Lipschitz continuous)  at $(\bar{\alpha}, \bar x)$  if there exist a neighborhood $\mathbb{V}$ of $\bar \alpha$, {a neighborhood $\mathbb{U}$ of $\bar x$} and $\kappa\geq 0$ such that
	{\rm{\begin{equation*}\label{calmness-defi}
			\Sigma \left(\alpha\right)\cap \mathbb{U} \subseteq \Sigma \left(\bar{\alpha}\right) + \kappa \left\| \alpha- \bar{\alpha}  \right\| B,\ \ \forall \alpha {\in \mathbb{V}}.
			\end{equation*}}}
\end{definition}
Note that the terminology of calmness was suggested by Rockafellar and Wets in \cite{RoWe98}. 

It is clear that both the pseudo-Lipschitz continuity and the upper-Lipschitz continuity are stronger than the pseudo upper-Lipschitz continuity. 
It is obvious that if $\Sigma: \R^{n}\rightarrow \R^{d}$ is a continuous single-valued map, then the pseudo-Lipschitz continuity  at  $(\bar{\alpha}, \bar x)$ reduces to the Lipschitz continuity at $\bar \alpha$, while   the calmness/pseudo upper-Lipschitz
continuity reduces to the calmness at $\bar \alpha$, i.e., there exist a neighborhood $\mathbb{V}$ of $\bar \alpha$ and a constant $\kappa\geq 0$ such that
$$ \| \Sigma(\alpha)-\Sigma(\bar \alpha)\|\leq \kappa \|\alpha-\bar \alpha\| \quad \forall \alpha {\in \mathbb{V}}.$$ 
Hence it is easy to see that the calmness/pseudo upper-Lipschitz
continuity is a much weaker stability condition than the  pseudo-Lipschitz continuity condition.

Many optimization problems can be written in the following form:
\begin{equation}
\min_x f(x) \quad \mbox{ s.t. } 0\in \Phi(x), \label{op} \end{equation}
where $f:\R^{d}\rightarrow \R$ is Lipschitz continuous around the point of interest and $\Phi:\R^{d}\rightrightarrows \R^{n}$ is a set-valued map with a closed graph.

Let $\bar x$ be a feasible solution for the above optimization problem. We say that  Mordukhovich (M-) stationary condition holds at $\bar x$  if  there exists $\eta$ such that
\begin{eqnarray}
&& 0\in \partial f(\bar x)+D^*\Phi(\bar x, 0)(\eta),\label{KKT}
\end{eqnarray}
respectively.
 
We now discuss the constraint qualifications under which a local optimal solution $\bar x$ satisfies the M-stationary condition.
For this purpose we consider the perturbed feasible solution mapping
\begin{equation} \Sigma(\alpha):=\Phi^{-1}(\alpha)=\{ x \in \R^d|  \alpha \in \Phi(x)\}.\label{feasible} \end{equation}
The property of the { calmness} of set-valued map $\Sigma(\cdot) $ at $( 0, \xb)\in \Gr \Sigma$ is equivalent with the property of the 
metric subregularity of its inverse map $\Sigma^{-1}(x)=\Phi(x)$ at $(\xb, 0)$, cf. \cite{DoRo04}. This justifies the terminology defined below.
\begin{definition} Let $0\in \Phi( \bar x)$. 
We say that 
  the metric subregularity constraint qualification (MSCQ) holds at $\bar x$ if the perturbed feasible solution mapping defined by (\ref{feasible}) is calm at $(0, \bar x)$.
\end{definition}
\begin{theorem}(Ye and Ye \cite[Theorem 3.1]{YeYe}) Let $\bar x $ be a local optimal solution of problem (\ref{op}). Suppose that  MSCQ holds at  $\bar x$. Then the M-stationary condition (\ref{KKT}) holds at $\bar x$. 
\end{theorem}
In the case where $\Phi(x):=(h(x), g(x)+\mathbb{R}_+^r)$
with $g:\mathbb{R}^n \rightarrow \mathbb{R}^r, h:\mathbb{R}^n \rightarrow 
 \mathbb{R}^s$ being  smooth, problem (\ref{op}) is the nonlinear program with equality and inequality constraints and (\ref{KKT}) reduces to the KKT condition for the nonlinear program. We say that a feasible solution $\bar x$ of problem (\ref{op}) satisfies the linear independence constraint qualification (LICQ) if 
 the gradients 
 $$  \{ \nabla g_i(\bar x)\}_{i\in I_g}  \cup \{ \nabla h_i(\bar x)\}_{i=1}^s , \quad  I_g=I_g(\bar x):=\{i| g_i(\bar x)=0\}$$
 are linearly independent. We say that the the positive linear independence constraint qualification (PLICQ) or no nonzero abnormal constraint qualification  (NNAMCQ) holds at $\bar x$ if there is no nonzero vector $(\eta^h,\eta^g)$ such that
 $$0= \sum_{i\in I_g} \nabla g_i(\bar x) \eta^g_i  +\sum_{i=1}^s \nabla h_i(\bar x) \eta^h_i,\quad \eta^g \geq 0 .$$
 By an alternative theorem, it is well-known that PLICQ/NNAMCQ is equivalent to the
 Mangasarian Fromovitz constraint qualification (MFCQ): 
 the gradients 
 $   \{ \nabla h_i(\bar x)\}_{i=1}^s $ is linearly independent and $\exists v\in \mathbb{R}^n $ such that
 \begin{equation*}
  \nabla g_i (\bar x)^T v<0 \  \forall i \in I_g,\quad
 \nabla h_i(\bar x)^Tv =0 \ \forall i =1,\dots, s .
 \end{equation*}
LICQ is stronger than MFCQ which is equivalent to saying that the perturbed feasible solution map $\Sigma(\cdot)$ is pseudo-Lipschitz continuous at $(0,\bar x)$ and hence stronger than the MSCQ/calmness condition.

We will also need the following definition.
\begin{definition}[Generalized linearity space] Given an arbitrary set $C\subseteq \R^d$, we call a subspace $L$ the generalized linearity space of $C$ and denote it by  $\Lsp(C)$ provided that it is   the largest subspace $L\subseteq \R^d$ such that
$C+L\subseteq C.$
\end{definition}
 In the case where $C$ is a convex cone,  the linearity space  of $C$ is the largest subspace contained in $C$ and can be calculated as $\Lsp(C)=(-C)\cap C$.

\section{Bilevel program with convex lower level program}
In this section we consider the case where given $x$ the lower level problem $({\rm P}_x)$ is a convex program. We first discuss the challenges for such a problem and follow by considering two special cases where the first one is a problem where the lower level problem is completely linear and the second one is a problem where the lower level constraint is independent of the upper level  variable.

To concentrate the main idea, for  simplicity  in this section we omit the upper level constraints and lower level equality constraints and consider 
\begin{eqnarray*}
({\rm BP})_1~~~~~~\min_{x,y} && F(x,y)\nonumber\\
{\rm s.t.} && y \in 
\arg\min_{y'} \{  f(x,y') |g(x,y')\leq 0\},
\end{eqnarray*}
where $g(x,y)$ is either affine in $y$ or convex in $y$ and the Slater condition holds, i.e., for each $x$ there is $y(x)$ such that $g(x,y(x))<0$. We assume that 
 $F:\mathbb{R}^n\times \mathbb{R}^m \rightarrow \mathbb{R}$, is continuously differentiable, 
 $f:\mathbb{R}^n\times \mathbb{R}^m \rightarrow \mathbb{R}$, $g:\mathbb{R}^n\times \mathbb{R}^m \rightarrow \mathbb{R}^r$ are twice continuously differentiable  in variable $y$.

Under the assumptions we made, the  KKT condition is necessary and sufficient for optimality. So 
$$y\in S(x) \Longleftrightarrow \exists \lambda \quad \mbox{ s.t. } 0=\nabla_yf(x,y)+\nabla_y g(x,y)^T \lambda, \quad 0\leq -g(x,y)\perp \lambda \geq 0.$$ 
 A common approach in the bilevel program literature is to replace ``$\exists \lambda$''   by ``$\forall \lambda$'' in the above and hence consider 
solving the following mathematical program with complementarity constraints (MPCC) instead. 
\begin{eqnarray*}
({\rm MPCC})~~~~~~\min_{x,y,\lambda} && F(x,y)\nonumber\\
{\rm s.t.} && 0=\nabla_yf(x,y)+\nabla_y g(x,y)^T \lambda,\\
&& 0\leq -g(x,y)\perp \lambda \geq 0.
\end{eqnarray*}
Problem (MPCC) looks like a standard mathematical program. However if one treats it as a mathematical program with equality and inequality constraints, then the usual constraint qualification such as MFCQ fails at each feasible solution (see Ye and Zhu \cite[Proposition 3.2]{YeZhuOp}). This observation leads to the introduction of weaker stationary conditions such as Weak (W-), Strong (S-), Mordukhovich (M-) and Clarke (C-) stationary conditions for (MPCC); see e.g. Ye \cite{Ye} for more discussions. We recall the definitions of various stationary conditions there.

\begin{definition}[Stationary conditions for MPCC]  Let $(\xb,\yb, \bar \lambda)$ be a feasible solution for  problem (MPCC). We say that 
$(\xb,\yb, \bar \lambda)$ is a weak stationary point of (MPCC) if there exist $w\in \R^m, \xi\in \R^r$ such that
  \begin{eqnarray}
  &&0=\nabla_x F(\xb,\yb)-\nabla_{yx}^2 f(\xb,\yb)w-\nabla^2_{yx}(\lb^Tg)(\xb, \yb)w+\nabla_xg(\xb,\yb)^T \xi,\nonumber \\
  &&0=\nabla_y F(\xb,\yb)-\nabla_{yy}^2 f(\xb,\yb)w-\nabla^2_{yy}(\lb^Tg)(\xb, \yb)w+\nabla_y
   g(\xb,\yb)^T\xi, \nonumber\\
  && \xi_i=0 \  \mbox{ if } g_i(\xb,\yb)<0, \lb_i=0, \label{EqOptCond-MPCC_c}\\
    && \nabla_y g_i(\xb,\yb)^T w=0 \ \mbox{ if } g_i(\xb,\yb)=0, \lb_i >0. \label{EqOptCond-MPCC_d}
  \end{eqnarray}
  We say that 
$(\xb,\yb, \bar \lambda)$ is a S-, M-, C- stationary point of (MPCC) if there exist $w\in \R^m, \xi\in \R^r$ such that the above conditions and  the following condition holds
\begin{eqnarray}
&& \xi_i\geq 0,   \nabla_y g_i(\xb,\yb)^T w\leq 0 \  \mbox{ if } g_i(\xb,\yb)=\lb_i=0,\nonumber\\
&&\mbox{either } \xi_i>0,   \nabla_y g_i(\xb,\yb)^T w<0  \mbox{ or } \xi_i\nabla_y g_i(\xb,\yb)^T w=0\  \mbox{ if } g_i(\xb,\yb)=\lb_i=0,\label{EqOptCond-MPCC_f} \\
&&  \xi_i\nabla_y g_i(\xb,\yb)^T w\leq 0\  \mbox{ if } g_i(\xb,\yb)=\lb_i=0, \nonumber 
\end{eqnarray}
respectively.
\end{definition}

For a mathematical program, it is well-known that under certain constraint qualification, a local optimal solution must be a stationary point and hence a stationary point is a candidate for a local optimal solution. Unfortunately as pointed out by Dempe and Dutta in \cite{Dam-Dut}, this is not true for bilevel programs even when the lower level is convex. Precisely, it is possible that $(\xb,\yb,\lb)$ is a local optimal solution of (MPCC) but  $(\xb,\yb)$ is not a local optimal solution of $(BP)_1$. Note that since (MPCC) is a nonconvex program, one usually only hope to find a local optimal solution and hence this is  very bad news. This observation  indicates that extreme care should be taken when  using MPCC reformulation in the case where the lower level problem has non-unique multipliers. 

\subsection{The bilevel program where the lower level program is completely linear}
We now discuss the special case of $(BP)_1$ where the lower level program is completely linear. That is, $f(x,y)=a^Tx+b^Ty$ and $g(x,y)=Cx+Dy-q$ with $a\in \mathbb{R}^n$, $b\in \mathbb{R}^m$, $C\in  \mathbb{R}^{r\times n}, D\in  \mathbb{R}^{r\times m}, q\in \mathbb{R}^r$. It is easy to see that  $(BP)_1$ can be equivalently written as the following problem
\begin{eqnarray*}
({\rm VP})_1~~~~~~\min && F(x,y)\nonumber\\
{\rm s.t.} &&a^Tx+b^Ty-V(x)\leq 0,\\
&&Cx+Dy-q\leq 0,
\end{eqnarray*}
where 
$V(x):=\displaystyle \inf_{y'} \{a^Tx+b^Ty'\mv  Cx+Dy'-q\leq 0\} $ is the value function of the lower level problem. Then by convex analysis, the value function V(x) is a polyhedral convex function and we have an explicit expression for its subgradient.
\begin{proposition}(see e.g., \cite[Proposition 4.1]{YeSYWu}) Let $\yb\in S(\xb)$ and suppose that 
$f(x,y)=a^Tx+b^Ty$ and $g(x,y)=Cx+Dy-q$. Then $V(x)$ is convex with  $ \partial V(\xb) \not =\emptyset$ and
$$ \partial V(\xb)=\left \{ a+C^T\nu| 0=b+D^T\nu ,  0\leq \nu\perp -(C \xb+D\yb-q) \geq 0 \right \}.$$
\end{proposition}
Since the function $a^Tx+b^Ty-V(x)$ is a concave function, it was shown in \cite{Ye04} that the nonsmooth weak reverse  constraint qualification holds for problem $({\rm VP})_1$ and hence by using the nonsmooth muliplier rule and the expression for the subgradient of the value function the following optimality condition holds.
\begin{theorem}(Ye \cite[Corollary 4.1]{Ye04}) Let $(\xb,\yb)$ be a local optimal solution of $(VP)_1$. Then there exists $\delta \geq 0, \bar \nu \in \mathbb{R}^r$ and 
$\alpha\in \mathbb{R}^r$ such that
  \begin{eqnarray*}
  &&0=\nabla_x F(\xb,\yb)+C^T (\alpha-\delta \bar \nu),\nonumber \\
  &&0=\nabla_y F(\xb,\yb)+D^T(\alpha-\delta \bar \nu), \nonumber\\
  && 0\leq \alpha \perp -(C \xb+D\yb-q) \geq 0,\\
  && 0\leq \bar \nu\perp -(C \xb+D\yb-q) \geq 0.
  \end{eqnarray*}
\end{theorem}
Let $\xi=\alpha-\delta \bar \nu$ where $\alpha, \delta, \bar \nu$ are those found in Theorem 2  and $w=0$ in Definition 9. It is easy to verify that $(\xb,\yb,\bar \nu)$ is an S-stationary point of the corresponding (MPCC). Unlike using (MPCC) reformulation by which usually a constraint qualification such as the MPCC LICQ is needed to ensure that a local optimal solution is an S-stationary point (see e.g. \cite{Ye}), $\bar \nu$ is a  multiplier for the lower level problem $(P_x)$ selected automatically from the subdifferential of the value function. For this particular multiplier 
$\bar \nu$,  the S-stationary condition holds under {\it no} constraint qualification. 

For bilevel programs where the lower level problem is not completely linear but  is convex with a convex  value function, the reader is referred to \cite{Ye04,Ye06} for more detailed discussions and results.

\subsection{The case where the lower level  constraint is independent of the upper level}
As we see in the previous discussion, the difficulty of using the first order approach occurs when the lower level has non-unique multipliers. In this subsection we consider a special case where the lower level constraint $\Gamma(x)=\Gamma$  is independent of $x$. Then $y\in S(x)$ if and only if  the generalized equation $0\in \nabla_y f(x,y)+N_\Gamma (y)$ holds.  So we next consider 
  the mathematical program with equilibrium constraints (MPEC) which is an equivalent reformulation of $(BP)_1$ when $g$ is independent of $x$:
\begin{eqnarray*}
({\rm MPEC})~~~~~~\min_{x,y} && F(x,y)\nonumber\\
{\rm s.t.} && 0\in \nabla_yf(x,y)+N_\Gamma(y),
\end{eqnarray*}
where $\Gamma:=\{y: g(y)\leq 0\}$ and $g:\mathbb{R}^m\rightarrow \mathbb{R}^r$ is either affine or convex with a Slater point.

For  problem (MPEC), Ye and Ye \cite{YeYe} showed that the pseudo upper-Lipschitz continuity/calmness/MSCQ   guarantees M-stationarity of solutions. 
\begin{theorem}(Ye and Ye \cite[Theorem 3.2]{YeYe})
Let $(\xb,\yb)$ be a local optimal solution of (MPEC). Suppose that the perturbed feasible solution map
$$\Sigma(\alpha):=\{(x,y)\mv \alpha \in  \nabla_yf(x,y)+N_\Gamma(y)\} $$
is  calm/pseudo upper-Lipschitz continuous at $(0, \xb,\yb)$ (i.e.,   MSCQ  holds at $(\bar x,\bar y)$).  Then $(\xb,\yb)$ is an M-stationary point of problem (MPEC), i.e.,
 there exist $w\in \R^m$ such that
  \begin{eqnarray*}
  &&0=\nabla_x F(\xb,\yb)+\nabla_{yx}^2 f(\xb,\yb)w,\\
  &&0=\nabla_y F(\xb,\yb)+\nabla_{yy}^2 f(\xb,\yb)w+D^* N_{\Gamma} (\yb, -\nabla_yf(\xb,\yb))(w).
  \end{eqnarray*}
\end{theorem}

In the case where  $\nabla_yf(x,y)$ is  affine  and $\Gamma$ is a convex polyhedral set, the set-valued map $ \Sigma(\cdot)$ is a polyhedral multifunction which means that its graph is the union of finitely many polyhedral convex sets. According to Robinson \cite{Robinson81}, $\Sigma (\cdot)$ is upper Lipschitz continuous  which implies that  MSCQ  holds automatically at each feasible solution. How to check MSCQ for the general case? 
We now describe a sufficient condition for MSCQ derived by Gfrerer and Ye in \cite{GfrYe16a}. First we start with some notation.
Let $$\bar\Lambda:=\{\lambda\mv  0=\nabla_y f(\bar x,\bar y) +\nabla g(\bar y)^T\lambda, 0\leq -g(\yb)\perp \lambda \geq 0\}$$
be the multiplier set for the lower level problem $(P_{\xb})$ at $\bar{y}$. Define the critical cone for $\Gamma$ at $\yb$  as 
$$\KbG:=\{v\mv \nabla g(\yb)v\in T_{\R^q_-}(g(\yb)), \nabla_y f(\xb,\yb)^Tv=0\}.$$ 
For every $v\in \KbG$, define the directional multiplier set at direction $v$ as 
$$\Lbv:=\argmax \left\{v^T\nabla^2(\lambda^Tg)(\yb)v\mv \lambda\in \Lb \right \}=\{\lambda| v^T\nabla^2g(\yb) v\in N_{\Lb}(\lambda)\}.$$
Let $\bar{I}:=\{i\mv g_i(\yb)=0\}$ be the index of constraints active at $\yb$. 
For every $\lambda\in \bar{\Lambda}$, we define the index set of strongly active constraints
$$\bar J^+(\lambda):= \{i| g_i(\yb)=0,\lambda_i>0\}.$$ 
Under our assumption of this section, the multiplier set $\bar\Lambda$ is nonempty and hence the critical cone for $\Gamma$ at $\yb$  can be represented as
$$\KbG=\Big\{v\mv \nabla g_i(\yb)^T v  \left \{
\begin{array}{ll} =0 & i\in \bar J^+(\Lb)\\
\leq 0 & i\in\bar I\setminus \bar J^+(\Lb)
\end{array} \right. \Big \}
\Big\} ,
$$ 
where $\displaystyle \bar J^+(\Lb):=\cup_{\lambda \in \Lb} \bar J^+(\lambda)$.
For every  $\bar v\in \KbG$, we denote the index set of active constraints for the critical cone at $\bar v$ as $$ \quad \bar{I}(\vb):=\{i\in \bar{I}|  \nabla g_i(\yb)^T \vb=0\}.$$
 Denote by $\bar {\cal E}$ the collection of all the extreme points of the closed and convex set of multipliers $\bar \Lambda$ and recall that $\lambda \in \bar \Lambda$ belongs to $\bar {\cal E}$ if and only the family of gradients $ \{ \nabla g_i(\bar y)| i\in \bar J^+(\lambda)\}$ is linearly independent.
Specializing the result from \cite {GfrYe16a} we have the following checkable constraint qualification for problem (MPEC). 
\begin{theorem}[{ \cite[Theorems 4 and  5]{GfrYe16a}}]
  Let $(\xb,\yb)$ be  a feasible solution of  problem (MPEC). Assume that
  there do not exist $(u,v)\not=0$, $\lambda\in \bar \Lambda(v)\cap \bar{\cal E}$ and $w\not=0$ satisfying
  \begin{eqnarray*}
    &&-\nabla_{yx}^2f(\xb,\yb)u-\nabla_{yy}^2f(\xb,\yb)v-\nabla^2 (\lambda^Tg)(\yb)v \in N_{\bar{K}_\Gamma}(v),\\
    \label{EqSuffMS3}&&\nabla_{xy}^2f(\xb,\yb)w=0,
    \\
    \label{EqSuffMS4}&&\nabla g_i(\yb)^Tw=0, i\in \bar{J} ^+(\lambda),\; w^T\left (\nabla_{yy}^2 f(\xb,\yb)+\nabla^2(\lambda^Tg)(\yb)\right )w
    \leq 0.\qquad
  \end{eqnarray*}
Then MSCQ for problem (MPEC) holds at $(\xb,\yb)$.
\end{theorem}

We would like to comment that recently \cite {Ad-Hen-Out} has compared the calmness condition for the two problems (MPEC) and (MPCC). They have shown that in general the calmness condition for (MPEC) is weaker than the one for the corresponding (MPCC).

Now we consider the M-stationary condition in Theorem 3. The expression of the M-stationary condition involves the coderivative of the normal cone mapping $N_{\Gamma} (\cdot)$. Precise formulae for this coderivative in terms of the problem data can be found in  \cite[Proposition 3.2]{HenrionW} if  $\Gamma$ is polyhderal, in  \cite[Theorem 3.1]{HenrionOS} if  LICQ  holds at $\yb$ for the lower level problem, in  \cite[Theorem 3]{GfrOut16a} under a relaxed MFCQ combined with the  so-called  $2-$regularity.

Recently Gfrerer and Ye \cite{GfrYe19} have derived a necessary optimality condition that is sharper than the M-stationary condition under   the following 2-nondegeneracy condition.
\begin{definition}\label{two-nond} Let  $v\in \KbG$. We say that $g$ is {\em 2-nondegenerate in direction $v$ at $\yb$} if
\[\nabla^2(\mu^Tg)(\yb)v \in N_\KbG(v)-N_\KbG(v),\ \mu\in \Span(\Lb(v)-\Lb(v))  \ \Longrightarrow\ \mu=0.\]
\end{definition}

\noindent In the case where the {directional} multiplier set $\Lb(v)$ is a singleton, $\Span(\Lb(v)-\Lb(v))=\{0\}$ and hence $g$ is  2-nondegenerate in {this} direction $v$.
\begin{theorem}(\cite[Theorem 6]{GfrYe19}) 
Assume that $(\xb,\yb)$ is a local minimizer for  problem (MPEC) fulfilling MSCQ at $(\xb,\yb)$. Further assume that $g$ is 2-nondegenerate in every nonzero critical  direction $0\not=v\in\KbG$.
Then there are a critical direction $\vb\in\KbG$, a directional multiplier $\lb\in\Lb(\vb)$, index sets $\J^+$, $\J$, ${\cal I}^+$, and ${\cal I}$ with $\bar J^+(\lb)\subseteq \J^+\subseteq\J\subseteq \bar J^+(\Lb(\vb))\subseteq \bar J^+(\Lb)\subseteq {\cal I}^+\subseteq {\cal I}\subseteq \bar I(\vb)$ and elements $w\in\R^m$, $\eta,\xi\in\R^q$ such that
  \begin{eqnarray*}
 &&0=\nabla_x F(\xb,\yb)-\nabla_{xy}^2 f(\xb,\yb)w,\\
 &&0=\nabla_y F(\xb,\yb)-\nabla_{yy}^2f(\xb,\yb)w -\nabla^2(\lb^Tg)(\yb)w+\nabla g(\yb)^T\xi+2\nabla ^2(\eta^Tg)(\yb)\vb,\\
 && \xi_i=0 \mbox{ if } i\not\in {\cal I},\\
&& \xi_i\geq 0, \nabla g_i(\yb)^T w\leq 0 \mbox{ if } i\in {\cal I}\setminus{\cal I}^+,\\
 && \nabla g_i(\yb)^T w=0 \mbox{ if } i\in {\cal I}^+,\\
  && \nabla g(\yb)^T\eta=0,\ \eta_i=0,i\not\in \J,\ \eta_i\geq 0, i\in \J\setminus \J^+.
  \end{eqnarray*}\end{theorem}

 In the case where the multiplier set $\Lb=\{\lb\}$ is a singleton, the 2-nondegeneracy condition holds automatically and the $\eta$ in the optimality condition becomes zero. In this case we have the following result.
\begin{corollary}(\cite[Corollary 1]{GfrYe19})\label{uniqm} \ Assume that $(\xb,\yb)$ is a local minimizer for  problem (MPEC) fulfilling MSCQ at $(\xb,\yb)$ and the lower level multiplier is unique, i.e., $\bar{\Lambda}=\{\lb\}$.
Then there are a critical direction $\vb\in\KbG$, index sets  ${\cal I}^+$
  with $\bar J^+(\lb)\subseteq {\cal I}^+\subseteq \bar I(\vb)$ and elements $w\in\R^m$, $\xi\in\R^q$ such that
  \begin{eqnarray}
 &&0=\nabla_x F(\xb,\yb)-\nabla_{xy}^2 f(\xb,\yb)w, \nonumber \\
 &&0=\nabla_y F(\xb,\yb)-\nabla_{yy}^2f(\xb,\yb)w -\nabla^2(\lb^Tg)(\yb)w+\nabla g(\yb)^T\xi,\nonumber\\
 && \xi_i=0 \mbox{ if } i\not\in \bar I(\vb), \label{newEqOptCondIndex_c}\\
&& \xi_i\geq 0, \nabla g_i(\yb)^T w\leq 0 \mbox{ if } i\in \bar I(\vb)\setminus{\cal I}^+,\nonumber \\
 && \nabla g_i(\yb)^T w=0 \mbox{ if } i\in {\cal I}^+.\label{newEqOptCondIndex_e}
  \end{eqnarray}
    \end{corollary}
 Actually we can show that the stationary condition in Theorem \ref{uniqm} is stronger than the M-stationary condition for (MPCC).  Suppose that $(\xb,\yb,\lb)$ satisfies the stationary condition in Theorem \ref{uniqm} and let $w\in\R^m$, $\xi\in\R^q$  be those found in Theorem  \ref{uniqm}. Then since 
 $$i\not \in  \bar I(\vb) \Longleftrightarrow  \begin{array}{ll}
\mbox{ either } g_i(\bar y)=0, \nabla g_i(\bar y)^T \bar v<0, \bar{\lambda}_i=0\\
\mbox{ or }  g_i(\bar y)<0, \bar{\lambda}_i=0 \end{array},$$
\eqref{newEqOptCondIndex_c} and \eqref{newEqOptCondIndex_e}
imply that   $\xi_i =0$ if $\bar{\lambda}_i=0$ and
$\nabla  g_i(\bar y) ^T w=0 $ if   $\bar \lambda_i >0$. 
It follows that \eqref{EqOptCond-MPCC_c}, \eqref{EqOptCond-MPCC_d}, \eqref{EqOptCond-MPCC_f} hold.
Therefore $(\xb,\yb, \bar \lambda )$ must satisfy the M-stationary condition for (MPCC) as well.
Hence in the case where the lower level multiplier is unique, the above stationary condition is in general stronger than the M-stationary condition of (MPCC). 

Finally in the rest of this section, we will discuss the S-stationary condition for (MPEC).
Let 
$$\bar{\cal N}:= \{ v\in \mathbb{R}^m\mv \nabla g_i(\yb)^T v=0 \  \forall i\in \bar{I}\}$$ be the nullspace of gradients of constraints active at $\yb$. Define for each $v\in \bar{\cal N}$, the sets
\begin{eqnarray*}
\bar{\cal W}(v)&:=&\left \{w\in \bar{K}_\Gamma \mv w^T \nabla^2 ((\lambda^1-\lambda^2)^T g)(\yb)v=0, \forall \lambda^1,\lambda^2 \in \bar{\Lambda}(v)\right \},\\
\tilde{\Lambda}(v)&:=& \left  \{\begin{array}{ll}
\bar{\Lambda}(v) \cap \bar{\cal E} & \mbox{ if } v\not =0\\
conv (\cup_{0\not =u \in \bar{K}_\Gamma} \bar{\Lambda}(u) \cap \bar{\cal E} ) &\mbox{ if } v=0, \bar{K}_\Gamma \not =\{0\},
\end{array} \right.,
\end{eqnarray*}
and for each $w\in \bar{K}_\Gamma $,
$$ \bar{L}(v;w):= \left \{ \begin{array}{ll}
\{-\nabla^2 (\lambda^T g)(\yb) w \mv \lambda \in \tilde{\Lambda}(v)\} +{\bar{K}_\Gamma}^\circ & \mbox{ if } \bar{K}_\Gamma \not = \{0\}\\
\mathbb{R}^m & \mbox{ if } \bar{K}_\Gamma  = \{0\}
\end{array} \right  . .$$ 
The following theorem is a slight improvement of \cite[Theorem 8]{GfrOut16b}  in that the assumption is weaker. 
\begin{theorem} Assume that $(\xb,\yb)$ is a local minimizer for  problem (MPEC) fulfilling MSCQ  at $(\xb,\yb)$ and the generalized linear independence constraint qualification holds:
\begin{equation*}\label{EqSuffS_Stat}\nabla P(\xb,\yb)\R^{n+m} +\Lsp\left(T_{\Gr N_\Gamma}\big(P(\xb,\yb)\big)\right)=\R^{2m}
,\end{equation*}
where 
$P(x,y):=(y, -\nabla_yf(x,y))$ and $\Lsp(C)$ is the generalized linearity space of set $C$ as defined in Definition 8.
Then $(\xb,\yb)$ is an S-stationary point for (MPEC), i.e., there exists elements $w $ such that
  \begin{eqnarray}
 &&0=\nabla_x F(\xb,\yb)+\nabla_{xy}^2 f(\xb,\yb)w, \nonumber \\
 &&0\in\nabla_y F(\xb,\yb)+\nabla_{yy}^2f(\xb,\yb)w +\widehat{D} {N}_\Gamma(\yb,-\nabla_y f(\xb,\yb))(w).\label{S-sc}
 \end{eqnarray}
  In particular, we have $w\in - \bigcap_{v\in \bar{\cal N}} \bar{\cal W}(v)$ and  
  \begin{eqnarray*}
 &&0=\nabla_x F(\xb,\yb)+\nabla_{xy}^2 f(\xb,\yb)w, \nonumber \\
 &&0\in\nabla_y F(\xb,\yb)+\nabla_{yy}^2f(\xb,\yb)w +\bigcap_{v\in \bar{I}(v) }\bar{L}(v;-w).
  \end{eqnarray*}
    \end{theorem}
\begin{proof} Since $(\xb,\yb)$ is a local minimizer for  problem (MPEC) which can be rewritten as 
\begin{eqnarray*}
({\rm MPEC})~~~~~~\min_{x,y} && F(x,y)\nonumber\\
{\rm s.t.} &&  P(x,y):=(y, -\nabla_yf(x,y))\in D:=\Gr N_\Gamma,
\end{eqnarray*}
by the basic optimality condition 
$$0\in \nabla F(\xb,\yb) +\widehat{N}_{\cal F}(\bar x,\bar y),$$
where ${\cal F}:=\{(x,y): P(x,y)\in D\}$.
By \cite[Theorem 4]{GfrOut16b}, under MSCQ and the generalized LICQ, 
$$\widehat{N}_{\cal F}(\bar x,\bar y)= \nabla P(\xb,\yb)^T \widehat N_D(P(\xb,\yb))$$
holds.  It follows that the S-stationary condition
$$0\in \nabla F(\xb,\yb)+\nabla P(\xb,\yb)^T \widehat N_D(P(\xb,\yb))$$ holds. Since
$$\nabla P(\xb,\yb)^T \widehat N_D(P(\xb,\yb))=\left \{ \left (\begin{array}{l}
\nabla_{xy}^2 f(\xb,\yb)w\\
\nabla_{yy}^2 f(\xb,\yb)w+w^*\end{array}\right ) | (w^*,-w) \in \widehat{N}_{\Gr N_\Gamma}(\bar y, -\nabla_yf(\bar x,\bar y)) \right \},
$$ and by definition of the co-derivative, 
$$ (w^*,-w) \in \widehat{N}_{\Gr N_\Gamma}(\bar y, -\nabla_yf(\bar x,\bar y)) \Longleftrightarrow w^* \in \widehat{D}  N_\Gamma(\bar y, -\nabla_yf(\bar x,\bar y))(w),$$
(\ref{S-sc}) follows.
By Gfrerer and Outrata \cite[Proposition  5]{GfrOut16b}, we have
\begin{eqnarray*}
\widehat{N}_D(P(\xb,\yb))&=& T_D(P(\xb,\yb))^\circ\\
&\subseteq &\left \{ (w^*,w)\mv w\in -\bigcap_{v\in \bar{\cal N}} \bar{\cal W}(v), w^*\in \bigcap_{v\in \bar{\cal N} }\bar{L}(v;-w) \right \}.
\end{eqnarray*}

\end{proof}

\section{Bilevel program with nonconvex lower level program}
In this section we consider the general  bilevel program (BP) as stated in the introduction and assume that 
 $F:\mathbb{R}^n\times \mathbb{R}^m \rightarrow \mathbb{R}$, $G:\mathbb{R}^n\times \mathbb{R}^m \rightarrow \mathbb{R}^p$ and  $H:\mathbb{R}^n\times \mathbb{R}^m \rightarrow \mathbb{R}^q$ are continuously differentiable, 
 $f:\mathbb{R}^n\times \mathbb{R}^m \rightarrow \mathbb{R}$, $g:\mathbb{R}^n\times \mathbb{R}^m \rightarrow \mathbb{R}^r, h:\mathbb{R}^n\times \mathbb{R}^m \rightarrow 
 \mathbb{R}^s$ are twice continuously differentiable  in variable $y$.

In the bilevel programming literature, in particular in early years, the  first order approach has been popularly used even when the lower level is nonconvex. 
But if the lower level program $(P_x)$ is not convex, 
the optimality condition $$0\in \nabla_y f(x,y) +\widehat{N}_{\Gamma(x)}(y)$$ is only  necessary but not sufficient for $y\in S(x)$. 
That is, the inclusion
$$S(x) \subseteq \left \{y| 0\in \nabla_y f(x,y) +\widehat{N}_{\Gamma(x)}(y)\right \}$$
may be strict. However,  It was pointed out by Mirrlees \cite{Mirrlees99} that an optimal solution of the bilevel program may not even be a stationary point of the reformulation by the first order approach.

Ye and Zhu \cite{YeZhuOp}  proposed to investigate the optimality condition  based on the value function reformulation first proposed by Outrata \cite{Outrata}. By the value function approach, one would  replace the original bilevel program  (BP) by the following equivalent problem:
\begin{eqnarray*}
({\rm VP})~~~~~~\min && F(x,y)\nonumber\\
{\rm s.t.} && f(x,y)-V(x)\leq 0,\\
&& g(x,y)\leq 0, h(x,y)=0,\\
&&  G(x,y)\leq 0,\ H(x,y)= 0,
\end{eqnarray*}
where 
$$V(x):=\inf_{y'} \{ f(x,y')\mv  g(x,y')\leq 0, h(x,y')=0\}$$ 
is  the value function of the lower level program.

There are two issues involved in using the value function approach. 
First, 
problem (VP) is a nonsmooth optimization problem since the value function $V(x)$ is in general nonsmooth. Moreover it is an implicit function of problem data. 
To ensure the lower semi-continuity of the value function we need the following assumption.
\begin{definition}(see \cite[Hypothesis 6.5.1]{Clarke} or \cite[Definition 3.8]{Lei-Lin-Ye-Zhang}  We say that the restricted inf-compactness holds around $\xb$ if $V(\xb)$ is finite and there exist a compact $\Omega$ and a positive number $\epsilon_0$ such that, for all $x\in \mathbb{B}_{\epsilon_0}(\xb)$ for which $V(x)<V(\xb)+\epsilon_0$, the problem $(P_x)$ has a solution in $\Omega$.
\end{definition}
The restricted inf-compactness condition is very weak. It does not even require the existence of solutions of problem $(P_x)$  for all $x $ near $\bar x$. A sufficient condition for the restricted inf-compactness to hold around $\bar x$ is the inf-compactness condition:
there exist $\alpha>0,\delta>0$ and a bounded set $C$ such that $\alpha >V(\bar x)$ and
$$\{ y | g(x,y) \leq 0, h(x,y)=0, f(x,y)\leq \alpha, x\in \mathbb{B}_\delta(\bar x)\} \subseteq C.$$


To ensure the Lipschitz continuity of the value function, we also need the following regularity condition.
\begin{definition}For $\yb\in S(\xb)$, we say that $(\xb,\yb)$ is quasi-normal if there is no nonzero vector $(\lambda^g,\lambda^h)$ such that 
$$ 0=\nabla g(\xb,\yb)^T \lambda^g+ \nabla h(\xb,\yb)^T \lambda^h,\quad  \lambda^g\geq 0$$
and there exists $(x^k,y^k)\rightarrow (\xb,\yb)$ such that
\begin{eqnarray*}
&& \lambda_i^g >0 \Rightarrow \lambda_i^g g_i(x^k,y^k) >0,\\
&& \lambda_i^h \not =0 \Rightarrow \lambda_i^h h_i(x^k,y^k) >0.
\end{eqnarray*}
\end{definition}
It is easy to see that the quasinormality is weaker than  MFCQ:
there is no nonzero vector $(\lambda^g,\lambda^h)$ such that 
\begin{eqnarray*}
&&0=\nabla g(\xb,\yb)^T \lambda^g+ \nabla h(\xb,\yb)^T \lambda^h,\\
&& 0\leq -g(\xb,\yb)\perp \lambda^g\geq 0.
\end{eqnarray*}
Now we can state a sufficient condition which ensures the Lipschitz continuity of the value function and an upper estimate for the limiting subdifferential of the value function. Since MFCQ is stronger than the quasi-normality and the set of quasi-normal multipliers is smaller than the classical multipliers, the following estimate is sharper and holds under weaker conditions than the classical counterpart in \cite[Corollary 1 of Theorem 6.5.2]{Clarke}.
\begin{proposition} \label{Prop2} \cite[Corollary 4.8]{Lei-Lin-Ye-Zhang} Assume that the restricted inf-compactness holds around $\xb$ and for each $\yb\in S(\xb)$, 
$(\xb,\yb)$ is quasi-normal. Then the value function $V(x)$ is Lipschitz continuous around $\xb$ with 
\begin{equation}
\partial V(\xb) \subseteq \widetilde W(\bar x) , \label{valuef} \end{equation}
where  
\begin{equation}\label{W}
\widetilde W(\bar x):=\bigcup_{\yb\in S(\xb)} \left \{ \nabla_x f(\xb,\yb)+\nabla_x g(\xb,\yb)^T \lambda^g+ \nabla_x h(\xb,\yb)^T \lambda^h: (\lambda^g,\lambda^h) \in {\cal M}(\xb,\yb)\right \},
\end{equation} where   $S(\bar x)$ denotes the solution set of the lower level program $(P_{\bar x})$
and ${\cal M}(\xb,\yb)$ is the set of quasi-normal multipliers, i.e., 
$$ {\cal M}(\xb,\yb) :=\left \{(\lambda^g,\lambda^h) \big | \begin{array}{l}
 0= \nabla_y f(\xb,\yb)+\nabla_y g(\xb,\yb)^T \lambda^g+ \nabla_y h(\xb,\yb)^T \lambda^h,\   \lambda^g\geq 0 \\
 \mbox{  there exists } (x^k,y^k)\rightarrow (\xb,\yb) \mbox{ such that } \\
 \lambda_i^g >0 \Rightarrow \lambda_i^g g_i(x^k,y^k) >0,\\
 \lambda_i^h \not =0 \Rightarrow \lambda_i^h h_i(x^k,y^k) >0 \end{array}  \right \} .$$
In addition to the above assumptions, if 
$\widetilde W(\bar x)=\{\zeta\},$ then $V(x)$ is strictly differentiable at $\bar x$ and $\nabla V(\bar x)=\{\zeta\}$.
\end{proposition}
Note that moreover if  the solution map of the lower level program $S(x)$ is semi-continuous at $(\bar x,\bar y)$ for some $\bar y\in S(\bar x)$, then the union $\bigcup_{\yb\in S(\xb)}$ sign can be omitted in (\ref{W}); see \cite[Corollary 1.109]{Mor}.

Secondly, is a local optimal solution of problem (BP) a stationary point of problem (VP)? For problem (VP), suppose that $(\xb,\yb)$ is a local optimal solution and if the value function $V(x)$ is Lipschitz continuous at $\xb$, then the  Fritz John type necessary optimality condition in terms of limiting subdifferential holds. That is, there exist multipliers   $\alpha \geq 0, \mu\geq 0, \lambda^g, \lambda^h, \lambda^G,\lambda^H$  not all equal to zero such that 
\begin{eqnarray*}
&& 0\in \alpha \nabla_x F(\bar x,\bar y) +\mu \partial_x( f-V)(\bar x,\bar y)\\
&& \quad +\nabla_x g(\xb,\yb)^T\lambda^g +\nabla_x h(\yb,\yb)^T\lambda^h+\nabla_x G(\yb,\yb)^T\lambda^G+\nabla_x H(\yb,\yb)^T\lambda^H   ,\\
&& 0\in \alpha \nabla_y F(\bar x,\bar y) +\mu \nabla_y f(\bar x,\bar y)\\
&& \quad +\nabla_y g(\yb,\yb)^T\lambda^g +\nabla_y h(\yb,\yb)^T\lambda^h+\nabla_y G(\yb,\yb)^T\lambda^G+\nabla_y H(\yb,\yb)^T\lambda^H   ,\\
&&  0\leq -g(\xb,\yb)\perp \lambda^g\geq 0, \quad  0\leq -G(\xb,\yb)\perp \lambda^G\geq 0.
\end{eqnarray*}
However it is easy to see that  every feasible solution $(x,y)$ of (VP) must be an optimal solution to the optimization problem
\begin{eqnarray*}
\min_{x',y'} && f(x',y')- V(x') \\
\mbox{ s.t. } &&  g(x',y') \leq 0, h(x',y')=0.
\end{eqnarray*} By Fritz-John type optimality condition, there exists $\tilde \mu\geq 0,\tilde \lambda^g, \tilde \lambda^h$ not all equal to zero such  that 
\begin{eqnarray*}
&& 0\in  \tilde \mu\partial( f-V)(x,y)+\nabla g(x,y)^T \tilde \lambda^g +\nabla h(x,y)^T \tilde \lambda^h\\
&&  0\leq -g(x,y)\perp \tilde \lambda^g\geq 0.
\end{eqnarray*} This means that there always exists a nonzero  abnormal multiplier $(0, \tilde \mu, \tilde \lambda^g, \tilde \lambda^h, 0,0)$ for the problem (VP) at each feasible solution, i.e., the no nonzero abnormal multiplier constraint qualification (NNAMCQ) fails at each feasible point of the problem (VP). Therefore unlike the standard nonlinear programs, we can not derive the KKT condition (i.e., the Fritz John condition when $\alpha=1$) from  lack of nonzero abnormal multipliers. As we can  see that the reason why NNAMCQ fails  is the existence of the value function constraint $f(x,y)-V(x) \leq 0$.
To address this issue,  Ye and Zhu \cite{YeZhuOp}  proposed the following partial calmness condition. 
\begin{definition}Let $(\bar x,\bar y)$ be a local  optimal  solution of problem (VP). We say that (VP) is partially calm at $(\bar x, \bar y)$ provided that there exist $\delta>0, \mu>0$ such that for all $\alpha \in\mathbb{B}_\delta$ and all $(x,y)\in \mathbb{B}_\delta(\bar x,\bar y)$ which are feasible for the partially perturbed problem
\begin{eqnarray*}
({\rm VP}_\alpha)~~~~~~\min && F(x,y)\nonumber\\
{\rm s.t.} && f(x,y)-V(x)+\alpha =0,\\
&& g(x,y)\leq 0, h(x,y)=0,\\
&&  G(x,y)\leq 0,\ H(x,y)= 0,
\end{eqnarray*} there holds
$F(x,y)-F(\bar x,\bar y)+\mu \|\alpha\| \geq 0.$
\end{definition}
It is obvious that the partial calmness is equivalent to the exact penalization, i.e., (VP) is partially calm at $(\bar x, \bar y)$ if and only if for some $\mu>0$,
$(\bar x, \bar y)$ is a local solution of the penalized problem 
\begin{eqnarray*}
(\widetilde{\rm VP})~~~~~~\min && F(x,y)+\mu (f(x,y)-V(x)) \nonumber\\
{\rm s.t.}  && g(x,y)\leq 0, h(x,y)=0,\\
&&  G(x,y)\leq 0,\ H(x,y)= 0.
\end{eqnarray*} 
Since the difficult constraint $f(x,y)-V(x) \leq 0$ is replaced by a penalty in the objective function, the usual constraint qualification such as MFCQ or equivalently NNAMCQ can be satisfied for problem $(\widetilde{\rm VP})$. Consequently using a nonsmooth multiplier rule
for problem $(\widetilde{\rm VP})$, one can derive a KKT type optimality condition for problem (VP). Such an approach has been used to derive necessary optimality condition of (BP) by Ye and Ye in \cite{YeZhuOp}  and later in other papers such as \cite{Dam-Dut-Mor,Dam-Zem,MorNamPhan}. For this approach to work, however, one needs to ensure the partial calmness condition.   
In  \cite{YeZhuOp}, it was shown that for the  minmax problem and the bilevel program where the lower level is completely linear, the partial calmness condition holds automatically. In \cite[Theorem 4.2]{Dam-Zem}, the last result was improved to conclude that the partial calmness condition holds automatically for any bilevel program where for each $x$, the lower level problem is a linear program.   In  \cite[Proposition 5.1]{YeZhuOp}, the uniform weak sharp mimimum is proposed as a sufficient condition for partial calmness and under certain conditions, the bilevel program with a quadratic program as the lower level program is shown to satisfy the partial calmness condition in in \cite[Proposition 5.2]{YeZhuOp} (with  correction in \cite{YeZhuOpcorrection}).

Apart from the issue of constraint qualification, we may ask a question on how likely an optimal solution of (VP) is a stationary point of (VP). 
In the case where there are no upper and lower level  constraints, the stationary condition of (VP) at $(\bar x,\bar y)$ means the
existence of $\mu \geq 0$ such that
\begin{eqnarray*}
&& 0\in \nabla_x F(\bar x,\bar y) +\mu \partial_x( f-V)(\bar x,\bar y),\\
&& 0=\nabla_y F(\bar x,\bar y)+\mu \nabla_y f(\xb,\yb).
\end{eqnarray*}
But this condition is very strong. It will not hold unless $0=\nabla_y F(\bar x,\bar y)$.

 As  suggested by Ye and Zhu in \cite{yz2},  we may consider the combined program
\begin{eqnarray*}
({\rm CP})~~~~~~\min_{x,y,u,v}  && F(x,y)\nonumber\\
{\rm s.t.} && f(x,y) -V(x)\leq 0,\\
&& 0=\nabla_y L(x,y,u,v):=\nabla_y f(x,y)+\nabla_y g(x,y)^Tu+\nabla_y h(x,y)^Tv ,\\
&& h(x,y)=0,\quad  0\leq -g(x,y)\perp u\geq 0,\\
 && G(x,y)\leq 0,\ H(x,y)= 0.
\end{eqnarray*}
 The motivation is clear since if the KKT conditions hold at each optimal solution of the lower level problem, then the KKT condition is a redundant condition. By adding the KKT condition we have not changed the feasible region of (BP). Note that this reformulation requires the existence of the KKT condition at the optimal solution of the lower level program; see \cite{Dam-Dut} for examples where the KKT condition does not hold at a lower level optimal solution.

Similarly as in the case of using MPCC to reformulate a bilevel program,  when the lower level multipliers are not unique, it is possible that $(\bar x, \bar y, \bar u,\bar v)$ is a local solution of (CP) but 
 $(\bar x, \bar y)$ is not a local optimal solution of (BP).
 
Due to the existence of the value function constraint $f(x,y)-V(x) \leq 0$, similarly to the analysis with problem (VP), 
 NNAMCQ will never hold at a feasible solution of (CP) and hence in  \cite{yz2}  the following partial calmness condition for problem (CP) is suggested as a condition to deal with the problem. 
 \begin{definition}Let $(\bar x,\bar y,\bar u,\bar v)$ be a local  optimal  solution of problem (CP). We say that (CP) is partially calm at $(\bar x, \bar y,\bar u,\bar v)$ provided that there exist $ \mu>0$ such that $(\bar x, \bar y,\bar u,\bar v)$ is a local solution of  the partially perturbed problem:
\begin{eqnarray*}
({\rm CP}_\mu)~~~~~~\min && F(x,y) +\mu (f(x,y)-V(x)) \nonumber\\
{\rm s.t.} 
&& 0=\nabla_y L(x,y,u,v),\\
&& h(x,y)=0,\quad  0\leq -g(x,y)\perp u\geq 0,\\
&&  G(x,y)\leq 0,\ H(x,y)= 0.
\end{eqnarray*}
\end{definition}
 Since there are more constraints in (CP) than in (VP), the partial calmness for (CP) is a weaker condition than the one for (VP).

 Given a feasible vector $(\bar x,\bar y,\bar u,\bar v)$ of the problem (CP),  define the following index sets:
 \begin{eqnarray*}
 && I_G=I_G(\bar x,\bar y):=\{i: G_i(\bar x,\bar y)=0\},\\
 &&I_g=I_g(\bar x,\bar y,\bar u):=\{i: g_i(\bar x,\bar y)=0, \bar u_i >0\},\\
 &&I_0=I_0(\bar x,\bar y,\bar u):=\{i: g_i(\bar x,\bar y)=0, \bar u_i =0\},\\
 &&I_u=I_u(\bar x,\bar y,\bar u):=\{i: g_i(\bar x,\bar y)<0, \bar u_i=0\}.
\end{eqnarray*}
 \begin{definition}[M-stationary condition for (CP) based on the value function]
A feasible point $(\bar x,\bar y,\bar u,\bar v)$ of problem (CP)  is called an
M-stationary point  based on the value function  if  there exist $\mu\geq 0$, $\beta\in \mathbb{R}^s$, $\lambda^G\in \mathbb{R}^p$, $\lambda^H\in \mathbb{R}^q$, $\lambda^g\in \mathbb{R}^m$, $\lambda^h\in \mathbb{R}^n$ such that the following conditions hold:
\begin{eqnarray*}
&&0\in \partial  F(\bar x,\bar y)+\mu \partial(f-V)(\bar x,\bar y)
+  \nabla G(\bar x,\bar y)^T\lambda^G+ \nabla H(\bar x,\bar y)^T \lambda^H \nonumber\\
&&~~~+  \nabla_{x,y} (\nabla_{y}L)(\bar x,\bar y)^T \beta
+ \nabla g(\bar x,\bar y)^T \lambda^g+ \nabla h(\bar x,\bar y)^T \lambda^h,\label{mpecn1}\\
&& \lambda_i^G \geq 0 \ \  i \in I_G,\ \  \lambda_i^G = 0 \ \  i \notin I_G,\\
&&\lambda_i^g=0 \ \  i\in I_u, \ (\nabla_y g(\bar x,\bar y)\beta)_i=0 \ \  i \in I_g, \label{mpecn2}\\
&&{either}\quad \lambda_i^g> 0, (\nabla_y g(\bar x,\bar y)\beta)_i>  0,\quad \mbox{or}\quad\lambda_i^g (\nabla_y g(\bar x,\bar y)\beta)_i=0 \quad    i \in I_0. \nonumber
\end{eqnarray*}
\end{definition}
In \cite[Theorem 4.1]{yz2}, it was shown that under the partial calmness condition and certain constraint qualifications,   a local optimal solution of (CP) must be an M-stationary  point based on the value function provided the value function is Lipschitz continuous. 
 
 Recently Xu and Ye \cite{XuYe}  introduced a nonsmooth version of the relaxed constant positive linear dependence (RCPLD) condition and apply it to (CP). We now describe the RCPLD condition.

In the following definition, we rewrite all equality constraints of problem (CP) by the equality constraint below:
$$0=\Phi(x,y,u,v):=\left (\begin{array}{c} \nabla_y L(x,y,u,v)\\
h(x,y)\\ H(x,y)
\end{array}  \right ).$$
We denote by $\{0\}^n$ the zero vector in $\mathbb{R}^n$ and by $e_i$  the unit vector with the $i$ th component  equal to $1$.
\begin{definition} Suppose that the value function $V(x)$ is Lipschitz continuous at $\bar x$.  Let $(\bar x,\bar y,\bar u,\bar v)$ be a feasible  solution of $(\rm CP)$. We say that  RCPLD holds at $(\bar x,\bar y,\bar u,\bar v)$ if the following conditions hold.
\begin{itemize}
\item[{\rm (I)}]
The vectors 
\begin{eqnarray*}
\{\nabla\Phi_i(x,y,u,v)\}_{i=1}^{m+s+q}\cup \{\nabla g_i(x,y)\times \{0^{r+s}\}\}_{i\in I_g}\cup \{(0^{n+m}, e_i,0^s)\}_{i\in I_u}
\end{eqnarray*}
 have the same rank 
 for all $(x,y,u,v)$ in a  neighbourhood of $(\bar x,\bar y,\bar u,\bar v)$.
\item[{\rm (II)}] \   Let  $\mathcal{I}_1\subseteq \{1,\cdots,m+s+q\}$, $\mathcal{I}_2\subseteq I_g$,  $\mathcal{I}_3\subseteq I_u$ be such that  the set of vectors $\{\nabla \Phi_i(\bar x,\bar y,\bar u,\bar v)\}_{i\in \mathcal{I}_1}\cup \{\nabla g_i(\bar x,\bar y)\times \{0\}\}_{i\in \mathcal{I}_2}\cup \{(0, e_i,0)\}_{i\in \mathcal{I}_3}$ 
 is a basis for  
 \begin{eqnarray*}
 &&  \mbox{ span }\large  \{\ \nabla \Phi_i(x,y,u,v)\}_{i=1}^{m+s+q}\cup \{\nabla g_i(\bar x,\bar y)\times \{0^{r+s}\}\}_{i\in I_g}\cup \{(0^{n+m}, e_i,0^s)\}_{i\in I_u} \large \}.\end{eqnarray*}
  For any index sets  $\mathcal{I}_4\subseteq I_G,\mathcal{I}_5,\mathcal{I}_6\subseteq I_0 $, the following  conditions hold.
  \begin{itemize}
  \item[(i)]
If  there exists a nonzero vector $(\lambda^V,\lambda^\Phi,\lambda^G, \lambda^g,\lambda^u) \in \mathbb{R}\times\mathbb{R}^{m+s+q}\times \mathbb{R}^{p}\times\mathbb{R}^r\times\mathbb{R}^r$ satisfying $\lambda^V\geq 0$, $\lambda^G \geq 0$
and $\mbox{either } \lambda_i^g> 0, \lambda_i^u>  0 \mbox{ or }\lambda_i^g\lambda_i^u=0, \forall i \in I_0$, $\xi^* \in \partial(f-V)(\bar x,\bar y)$  such that 
\begin{eqnarray*}
&&0 =  \lambda^V \xi ^* +  \sum_{i\in \mathcal{I}_1} \lambda_i^\Phi \nabla \Phi_i(\bar x,\bar y,\bar u,\bar v) +\sum_{i\in\mathcal{I}_4} \lambda_i^G \nabla G_i(\bar x,\bar y)\times \{0^{r+s}\} \\
&&
 +\sum_{i\in\mathcal{I}_2\cup\mathcal{I}_5} \lambda_i^g \nabla g_i(\bar x,\bar y)\times \{0^{r+s}\}
 -\sum_{i\in\mathcal{I}_3\cup\mathcal{I}_6}\lambda_i^u  (0^{n+m}, e_i, 0^s)
\end{eqnarray*} 
and $(x^k,y^k,u^k,v^k, \xi^k)\rightarrow (\bar x,\bar y,\bar u,\bar v,\xi^*)$ as $k\rightarrow \infty$, $\xi^k \in \partial (f-V)(x^k, y^k)$
then the set of vectors 
\begin{eqnarray*}
&& \{\xi^k\}\cup \{\nabla \Phi_i(x^k,y^k,u^k,v^k)\}_{i\in \mathcal{I}_1}\cup \{\nabla G_i(x^k,y^k)\times \{0^{r+s}\} \}_{i\in \mathcal{I}_4} \\
&&  \cup \{\nabla g_i(x^k,y^k)\times \{0^{r+s}\}\}_{i\in \mathcal{I}_2\cup\mathcal{I}_5}\cup\{(0^{n+m}, e_i, 0^s)\}_{i\in \mathcal{I}_3\cup\mathcal{I}_6},
\end{eqnarray*} where $k$ is sufficiently large and $(x^k,y^k, u^k,v^k) \neq (\bar x, \bar y, \bar u, \bar v)$,
is linearly dependent. 
\item[(ii)]\  If  there exists a nonzero vector $(\lambda^\Phi,\lambda^G, \lambda^g,\lambda^u) \in \mathbb{R}^{m+s+q}\times \mathbb{R}^{p}\times\mathbb{R}^r\times\mathbb{R}^r$ satisfying  $\lambda^G \geq 0$
and $\mbox{either } \lambda_i^g> 0, \lambda_i^u>  0 \mbox{ or }\lambda_i^g\lambda_i^u=0, \forall i \in I_0$ such that 
\begin{eqnarray*}
&&0 =   \sum_{i\in \mathcal{I}_1} \lambda_i^\Phi \nabla \Phi_i(\bar x,\bar y,\bar u,\bar v) +\sum_{i\in\mathcal{I}_4} \lambda_i^G \nabla G_i(\bar x,\bar y)\times \{0^{r+s}\} \\
&&
 +\sum_{i\in\mathcal{I}_2\cup\mathcal{I}_5} \lambda_i^g \nabla g_i(\bar x,\bar y)\times \{0^{r+s}\}
 -\sum_{i\in\mathcal{I}_3\cup\mathcal{I}_6}\lambda_i^u  (0^{n+m}, e_i, 0^s),
\end{eqnarray*} 
and  $(x^k,y^k,u^k,v^k)\rightarrow (\bar x,\bar y,\bar u,\bar v)$, as $k\rightarrow \infty$.
Then the set of vectors 
\begin{eqnarray*}
&&  \{\nabla \Phi_i(x^k,y^k,u^k,v^k)\}_{i\in \mathcal{I}_1}\cup \{\nabla G_i(x^k,y^k)\times \{0^{r+s}\} \}_{i\in \mathcal{I}_4} \\
&&  \cup \{\nabla g_i(x^k,y^k)\times \{0^{r+s}\}\}_{i\in \mathcal{I}_2\cup\mathcal{I}_5}\cup\{(0^{n+m}, e_i, 0^s)\}_{i\in \mathcal{I}_3\cup\mathcal{I}_6},
\end{eqnarray*} where $k$ is sufficiently large and $(x^k,y^k, u^k,v^k) \neq (\bar x, \bar y, \bar u, \bar v)$,
is linearly dependent.
\end{itemize}

\end{itemize}
\end{definition} Since 
\begin{eqnarray*}
\partial(f-V)(\bar x,\bar y) & \subseteq &\partial^c(f-V)(\bar x,\bar y)\\
&  =& \nabla f(\bar x,\bar y) -\partial^c V(\bar x )\times \{0\} \subseteq  \nabla f(\bar x,\bar y)  -conv \widetilde W(\bar x)\times \{0\},\end{eqnarray*}
where $\widetilde W(\bar x)$ is the upper estimate of the limiting subdifferential of the value function at $\bar x$ defined as in  (\ref{W}),
 we can replace the set $\partial(f-V)(\bar x,\bar y)$ by its upper estimate $\nabla f(\bar x,\bar y)  -conv \widetilde W(\bar x)\times \{0\}$ in RCPLD and obtain a sufficient condition for RCPLD. Moreover if  the solution map of the lower level program $S(x)$ is semi-continuous at $(\bar x,\bar y)$, then the set $\partial(f-V)(\bar x,\bar y)$ can be replaced by its upper estimate $\nabla f(\bar x,\bar y)  - \widetilde W(\bar x)\times \{0\}$.

\begin{theorem} \cite{XuYe}Let $(\bar x,\bar y,\bar u,\bar v)$ be a local solution of $(\rm CP)$ and suppose that the value function $V(x)$ is Lipschitz continuous at $\bar x$.
If 
{RCPLD holds at } $(\bar x,\bar y,\bar u,\bar v)$, then 
$(\bar x,\bar y,\bar u,\bar v)$ is an M-stationary point of problem (CP) based on the value function. 
\end{theorem} 
In the following result, the value function constraint $f(x,y)-V(x)\leq 0$ is not needed in the verification.
\begin{theorem}\cite{XuYe} Let $(\bar x,\bar y,\bar u,\bar v)$ be a local solution of $(\rm CP)$ and suppose that the value function $V(x)$ is Lipschitz continuous at $\bar x$.
If the rank of the matrix 
$$J^*=\left [\begin{array}{cccc}
\nabla  (\nabla_y f+\nabla_y g^T\bar u+\nabla_y h^T\bar v)(\bar x,\bar y) & \nabla_y h(\bar x,\bar y)^T & \nabla_{y} g_{I_g\cup I_0}(\bar x,\bar y)^T \\
\nabla h(\bar x,\bar y)  &0 &0\\
\nabla H(\bar x,\bar y)  &0&0 \\
\nabla g_{I_g}(\bar x,\bar y)  & 0&0
\end{array}\right ]$$
is equal to $m+n+r+s-|I_u|$. 
Then 
{RCPLD holds and}
$(\bar x,\bar y,\bar u,\bar v)$ is an M-stationary point of problem (CP) based on the value function. 
\end{theorem}

In the last part of this section we briefly summarize some necessary optimality conditions obtained in \cite{y11} using the combined approach.
For any given $\bar x$,  define the set
\begin{eqnarray*}
\lefteqn{ W(\bar x)}\\
&&:=\bigcup_{\yb\in S(\xb)} \left \{ \nabla_x f(\xb,\yb)+\nabla_x g(\xb,\yb)^T \lambda^g+ \nabla_x h(\xb,\yb)^T \lambda^h: \begin{array}{l}
0=\nabla_y g(\xb,\yb)^T \lambda^g+ \nabla_y h(\xb,\yb)^T \lambda^h\\
  0\leq -g(x,y)\perp \lambda^g \geq 0
  \end{array}  \right \}.
\end{eqnarray*}
It is easy to see that $ \widetilde W(\bar x) \subseteq W(\bar x)$ and under the assumption made in Proposition \ref{Prop2}, it  is an upper estimate of the limiting subdifferential of the value function. 
 \begin{definition} Let $(\bar x,\bar y, \bar u, \bar v)$ be a feasible solution to (CP). We say that (CP) is weakly calm at $(\bar x,\bar y, \bar u, \bar v)$ with modulus $\mu>0$ if
 $$ [\nabla F(\bar x,\bar y)+\mu \nabla f(\bar x,\bar y)]^T (d_x,d_y)-\mu  \min_{ \xi \in  W(\bar x)} \xi d_x \geq 0 \quad \forall d \in {\cal L}^{MPEC} ((\bar x,\bar y, \bar u, \bar v); \widetilde{F}), $$
 where $ \widetilde{F}$ is the feasible region of problem $({\rm CP}_\mu)$ and ${\cal L}^{MPEC} ((\bar x,\bar y, \bar u, \bar v); \widetilde{F})$ is the MPEC linearized cone of $ \widetilde{F}$ defined by
 \begin{eqnarray*}
 \lefteqn{{\cal L}^{MPEC} ((\bar x,\bar y, \bar u, \bar v); \widetilde{F})}\\
 &&:=\left \{ (d_x,d_y,d_u,d_v) |\begin{array}{ll}
 \nabla_{x,y} (\nabla_y L)(\bar x,\bar y,\bar u,\bar v) (d_x,d_y) +\nabla_y g(\bar x,\bar y)^T d_u+\nabla_y h(\bar x,\bar y)^T d_v=0 &\\
 \nabla G_i (\bar x,\bar y)^T (d_x,d_y) \leq 0, & i\in I_G\\
  \nabla H_i (\bar x,\bar y)^T (d_x,d_y) = 0,\\
 \nabla g_i(\bar x,\bar y)^T (d_x,d_y) =0,& i\in I_g\\
  (d_u)_i=0, &  i\in I_u \\
   \nabla g_i(\bar x,\bar y)^T (d_x,d_y) \cdot (d_u)_i =0,  \nabla g_i(\bar x,\bar y)^T (d_x,d_y) \leq 0, (d_u)_i\geq 0 & i \in I_0.
 \end{array}
  \right \}.
 \end{eqnarray*}
 
 \end{definition} 
 \begin{definition}[M-stationary condition for (CP) based on an upper estimate]
A feasible point $(\bar x,\bar y,\bar u,\bar v)$ of problem (CP)  is called an
M-stationary point based on  an upper estimate  if there exist $\mu\geq 0$, $\beta\in \mathbb{R}^s$, $\lambda^G\in \mathbb{R}^p$, $\lambda^H\in \mathbb{R}^q$, $\lambda^g\in \mathbb{R}^m$, $\lambda^h\in \mathbb{R}^n$ such that the following conditions hold:
\begin{eqnarray*}
&&0\in \partial  F(\bar x,\bar y)+\mu [\nabla f(\bar x,\bar y) -conv W(\bar x)\times \{0\}]
+  \nabla G(\bar x,\bar y)^T\lambda^G+ \nabla H(\bar x,\bar y)^T \lambda^H \nonumber\\
&&~~~+  \nabla (\nabla_{y} f+ \nabla_{y} g^T \bar u 
+\nabla_{y} h^T \bar v )(\bar x,\bar y)^T \beta
+ \nabla g(\bar x,\bar y)^T \lambda^g+ \nabla h(\bar x,\bar y)^T \lambda^h,\\
&& \lambda_i^G \geq 0 \ \  i \in I_G,\ \  \lambda_i^G = 0 \ \  i \notin I_G,\\
&&\lambda_i^g=0 \ \  i\in I_u, \ (\nabla_y g(\bar x,\bar y)\beta)_i=0 \ \  i \in I_g,\\
&&{either}\quad \lambda_i^g> 0, (\nabla_y g(\bar x,\bar y)\beta)_i>  0,\quad \mbox{or}\quad\lambda_i^g (\nabla_y g(\bar x,\bar y)\beta)_i=0 \quad    i \in I_0. \nonumber
\end{eqnarray*}
\end{definition}
Let $(\bar x,\bar y,\bar u,\bar v)$ be a feasible solution of problem (CP). By \cite[Theorem 4.3]{y11}, if the set $ W(\bar x)$ is nonempty and compact and (CP) is MPEC-weakly calm at $(\bar x,\bar y,\bar u,\bar v)$, then $(\bar x,\bar y,\bar u,\bar v)$ is an M-stationary point of problem (CP)  based on  an upper estimate. Note that it is obvious that the M-stationary condition based on an upper estimate is weaker than the corresponding M-stationary condition based on the value function.
\begin{acknowledgement}
The research of this author was partially
supported by NSERC.  The author would like to thank an  anonymous referee for the helpful suggestions and comments that have helped to improve the presentation of the paper.
\end{acknowledgement}

\begin{thebibliography}{99.}%
%
%
%

\bibitem{Ad-Hen-Out} {\sc L. Adam, R. Henrion and J. Outrata}, {\em On M-stationarity conditions in MPECs and the associated qualification conditions}, Math. Program., 168 (2018), pp. 229-259.
%
\bibitem{Aubin1984Lipschitz}
{\sc J. Aubin}, {\em Lipschitz behavior of solutions to convex minimization problems},   Math. Opera.  Res.,  9 (1984), pp. 87-111.

\bibitem{Clarke} {\sc F.H. Clarke}, {\em Optimization and Nonsmooth Analysis}, Wiley-Interscience, New York, 1983.

\bibitem{d1} {\sc S. Dempe}, {\em Foundations of Bilevel Programming}, Kluwer Academic Publishers, 2002.
\bibitem{d2} {\sc S. Dempe}, {\em Annotated bibliography on bilevel programming and mathematical programs with equilibrium constraints}, Optimization,  { 52} (2003),
pp. 333-359.
%
\bibitem{Dam-Dut} {\sc S. Dempe and  J. Dutta}, {\em Is bilevel programming a special case of a mathematical program with complementarity constraints?}, Math. Program., 131 (2012), pp. 37-48.
\bibitem{Dam-Dut-Mor} {\sc S. Dempe, J. Dutta and B.S. Mordukhovich}, {\em New necessary optimality conditions in optimistic bilevel programming}, Optimization, 56 (2007), pp. 577-604.

\bibitem{Dam-Zem} {\sc S. Dempe and A.B. Zemkoho}, {\em The bilevel programming problems: reformulations, constraint qualifications and optimality conditions}, Math. Program.,  138 (2013), pp.~447-473.
%
\bibitem{DoRo04} {\sc A.L. Dontchev and R.~T. Rockafellar}, {\em Regularity and conditioning of solution mappings in variational anlysis},
    Set-Valued Anal., 12 (2004), pp. 79-109.
%
%
%
%
%
%
%
%
\bibitem{GfrOut16a}{\sc H. Gfrerer and   J.V. Outrata}, {\em On computation of limiting coderivatives of the normal-cone mapping to inequality
systems and their applications}, Optimization,  65 (2016), pp. 671--700.
%
%
\bibitem{GfrOut16b}{\sc H. Gfrerer and J.V. Outrata}, {\em  On computation of generalized derivatives of the normal-cone mapping and their applications},   Math. Oper. Res. 41 (2016), pp.~1535--1556.
%
\bibitem{GfrYe16a}{\sc H. Gfrerer and  J.J. Ye}, {\em New constraint qualifications for
    mathematical programs with equilibrium constraints via variational analysis},  SIAM J. Optim., 27 (2017), pp. 842-865.
%
%
\bibitem{GfrYe19}{\sc H. Gfrerer and  J.J. Ye}, {\em New sharp necessary optimality conditions for
    mathematical programs with equilibrium constraints}, to appear in  Set-Valued Var.  Anal. 
  
  \bibitem{HenrionOS} {\sc R. Henrion, J. Outrata and T. Surowiec}, {\em On the co-derivative of normal cone mappings to inequality systems}, Nonlinear Anal. Theo. Meth.  Appl., 71 (2009), pp. 1213-1226.
  \bibitem{HenrionW} {\sc R. Henrion and W. R\"{o}misch}, {\em On M-stationary points for a stochastic equilibrium problem under equilibrium constraints in electricity spot market modeling}, Appl. Math. 52 (2007), pp. 473-494.
     \bibitem{Lei-Lin-Ye-Zhang} {\sc L.  Guo, G-H. Lin, J.J. Ye and J. Zhang}, {\em Sensitivity analysis of the value functions for parametric mathematical programs with equilibrium constraints}, SIAM J. Optim., 24 (2014), pp. 1206-1237.
     
%
%
    
    \bibitem{K}{\sc G. Kunapuli, K. P. Bennett, J. Hu and J-S. Pang}, {\em Classification model selection via bilevel programming}, Optim. Meth. Software, 23 (2008), pp. 475-489.
    
    \bibitem{L} {\sc Y-C. Lee, J-S. Pang and J. E. Mitchell}, {\em Global resolution of the support vector machine regression parameters selection problem with LPCC}, EURO J. Comput. Optim., 3(2015), pp. 197-261.
   
    \bibitem{Luo-Pang-Ralph} {\sc Z-Q. Luo, J-S. Pang}, {\em Mathematical Programs with Equilibrium Constraints}, Cambridge University Press, 1996.
    \bibitem{Mirrlees99}  {\sc J. Mirrlees},  {\em The theory of moral hazard and unobservable behaviour-- part I},  Review of Economic Studies, { 66} (1999), pp. 3-22.
    \bibitem{Mor} {\sc B.S. Mordukhovich}, {\em Variational Analysis and Generalized Differentiation, Vol. 1: Basic Theory, Vol. 2: Applications}, Springer, Berlin, 2006.

\bibitem{MorNamPhan} {\sc B.S. Mordukhovich, N. M. Nam and H. M. Phan}, {\em Variational analysis of marginal functions with applications to bilevel programming}, 
J. Optim. Theory  Appl., 152 (2012), pp. 557-586.

\bibitem{Outrata} {\sc J.V. Outrata}, {\em On the numerical solution of a class of Stackelberg problems}, ZOR-Math.  Methods  Oper. Res., 
34(1990), pp.255-277.
    \bibitem{Out-Koc-Zowe} {\sc J.V. Outrata and M. Ko\v{c}vara, J. Zowe}, {\em Nonsmooth Approach to Optimization Problems with Equilibrium  Constraints: Theory, Applications and Numerical Results}, Kluwer Academic Publishers, Dordrecht, The Netherlands, 1998.
%
%
%
\bibitem{Robinson1975Stability}
{\sc S.M. Robinson}, {\em Stability theory for systems of inequalities. {P}art {I}: Linear systems},  SIAM J. Numer.  Anal. 12 (1975), pp. 754-769.

 \bibitem{Robinson81} {\sc S.M. Robinson}, {\em Some continuity properties of polyhedral multifunctions}, Math. Program. Stud., 14 (1981), pp. 206-214.	
%
\bibitem{RoWe98}{\sc R.T. Rockafellar, R.~J-B. Wets}, {\em Variational Analysis}, Springer, Berlin, 1998.

\bibitem{v} H. von Stackelberg, {\em Marktform and Gleichgewicht} Springer-Verlag, Berlin, 1934. engl. transl.: The Theory of the Market Economy, Oxford University Press, Oxford, England, 1954.

%
%
\bibitem{XuYe} {\sc M. Xu and J.J. Ye}, {\em Relaxed constant positive linear dependence constraint qualification and its application to bilevel programs}, Preprint.

\bibitem{Ye04} {\sc J.J. Ye}, {\em Nondifferentiable multiplier rules for  optimization and bilevel optimization problems}, SIAM J. Optim., 15 (2004), pp. 252-274.
\bibitem{Ye} {\sc J.J. Ye}, {\em Necessary and sufficient optimality conditions for optimization programs with equilibrium constraints}, J. Math. Anal. Appl., 307 (2005), pp. 350-369.
\bibitem{Ye06} {\sc J.J. Ye}, {\em Constraint qualifications and KKT conditions for bilevel programming problems}, Math. Oper. Res., 
{ 31} (2006), pp. 811-824.



\bibitem{y11}J.J. Ye,   {\em Necessary optimality conditions for multiobjective bilevel programs},   Math. Oper. Res., {36} (2011), pp. 165--184.
\bibitem{YeSYWu}{\sc J.J. Ye and S.Y. Wu}, {\em First order optimality conditions for generalized semi-infinite programming problems}, J. Optim. Theory  Appl.,  137 (2008), pp. 419-434.
\bibitem{YeYe}{\sc J.J. Ye and X.Y. Ye}, {\em Necessary optimality conditions for optimization problems with variational inequality constraints},  Math. Oper. Res., 22 (1997), pp. 977-997.
\bibitem{YeZhuOp} {\sc J.J. Ye and D.L. Zhu}, {\em Optimality conditions for bilevel programming problems}, Optimization, 33 (1995), pp. 9-27.
\bibitem{YeZhuOpcorrection} {\sc J.J. Ye and D.L. Zhu}, {\em A note on optimality conditions for bilevel programming problems}, Optimization, 39 (1997), pp. 361-366.
 

 \bibitem{yz2} J.J. Ye and D.L.  Zhu, {\em New necessary optimality conditions for bilevel programs by combining the MPEC and value function approaches}, SIAM J. Optim., {20} (2010), pp. 1885-1905.

\end{thebibliography}
\end{document}